\documentclass[11pt]{article}
\usepackage[sort&compress,numbers,comma,square]{natbib}
\usepackage{amsmath, amsthm, amssymb, amsfonts, enumitem, verbatim,
  xspace, graphics}
\usepackage[active]{srcltx}
\usepackage[colorlinks]{hyperref}
\long\def\comment#1{}

\newtheorem{theorem}{Theorem}
\newtheorem{prop}[theorem]{Proposition}

\newtheorem{lemma}[theorem]{Lemma}

\theoremstyle{definition}

\theoremstyle{remark}

\newlist{enum}{enumerate}{10}
\setlist[enum]{itemsep=0ex, topsep=1ex, leftmargin=4ex, label={\arabic*})}

\makeatletter
\def\as{a.s\@ifnextchar.{}{.\xspace}}
\def\iid{i.i.d\@ifnextchar.{}{.\xspace}}
\makeatother

\def\Erdos{Erd\H{o}s\xspace}

\def\Renyi{R\'enyi\xspace}

\def\ai{\varepsilon}  
\def\Cal#1{\mathcal{#1}}
\def\Cbr#1{\left\{#1\right\}}
\def\cf#1{\mathbf{1}\!\Cbr{#1}}
\def\cnt{\mathop{n}}
\def\cov{\mathop{\rm Cov}}
\def\dbern{\mathrm{Bern}}

\def\dpois{\mathrm{Po}}
\def\drg{\mathrm{G}}
\def\Empty{\emptyset}
\def\eno#1#2{#1_1, \ldots, #1_{#2}}             
\def\esd#1{\mu_{#1}}
\def\ess{\mathop{\rm ess}}
\def\ev{\lambda}
\def\Grp#1{\left(#1\right)}
\def\gv{\,|\,}
\def\Iff{\Longleftrightarrow}   
\def\ismat#1{\in\Cal M_#1}
\def\lfrac#1#2{#1/#2}
\def\mean{\mathop{\rm E}}
\def\mnt{\beta}
\def\pr{\mathop{\rm P}}
\def\Reals{\mathbb{R}}
\def\Sbr#1{\left[#1\right]}
\def\sev{\ev_\star}       
\def\sr{\varrho}  
\def\th{{\rm th}}
\def\toi{\to\infty}
\def\tr{\mathop{\rm tr}}
\def\var{\mathop{\rm Var}}
\def\vf#1{\boldsymbol{#1}}

\begin{document}
\begin{center}
  \textbf{
    Random reversible Markov matrices with tunable extremal
    eigenvalues
  } \\[1ex]
  \normalsize
  Zhiyi Chi \\
  Department of Statistics, University of Connecticut \\
  215 Glenbrook Road, U-4120, Storrs, CT 06269, USA\\
  Email: zhiyi.chi@uconn.edu.
  \\[.8ex]
  \today
\end{center}

\begin{abstract}
  Random sampling of large Markov matrices with a tunable spectral
  gap, a nonuniform stationary distribution, and a nondegenerate
  limiting empirical spectral distribution (ESD) is useful.  Fix $c>0$
  and $p>0$.  Let $A_n$ be the adjacency matrix of a random graph
  following $\mathrm{G}(n, p/n)$, known as the \Erdos-\Renyi
  distribution.  Add $c/n$ to each entry of $A_n$ and then normalize
  its rows.  It is shown that the resulting Markov matrix has the
  desired properties.  Its ESD weakly converges in probability to a
  symmetric nondegenerate distribution, and its extremal eigenvalues,
  other than 1, fall in $[-1/\sqrt{1+c/k},-b]\cup [b,1/\sqrt{1+c/k}]$
  for any $0<b<1/\sqrt{1+c}$, where $k = \lfloor p \rfloor + 1$.
  Thus, for $p\in (0,1)$, the spectral gap tends to $1-1/\sqrt{1+c}$.

  \medbreak\noindent

  \medbreak\noindent
\end{abstract}

\section{Introduction} \label{s:intro}
The spectral properties of random Markov matrices have received
increasing attention over the years \cite {tran:13:rsa,
  bordenave:11:ap, bordenave:10:alea, nguyen:14:ap, bordenave:12:ptrf,
  chatterjee:15:aihp}.  In applications, it is useful to randomly
sample a large Markov matrix, such that the mixing rate of the
associated Markov chain is controllable.  The chain can be used,
for example, to evaluate the performance of a data analytic procedure
under various strengths of statistical dependency within data
\cite {rayaprolu:14:arXiv}.  By the well known connection between
mixing rate and eigenvalues of Markov matrix \cite {rosenthal:95:siam,
  chung:97:ams}, the issue may be cast as how to sample large Markov
matrices with a specified spectral gap.  This note addresses the
issue for reversible Markov matrices.

Denote by $\Cal M_n$ the set of $n\times n$ matrices with all entries
being nonnegative.  For $M\ismat n$, if its eigenvalues are
$\ev_1(M)$, \ldots, $\ev_n(M)$, counting multiplicity, then its
spectral radius is $\sr(M) = \max|\ev_i(M)|$ and its empirical
spectral distribution (ESD) is
\begin{align*}
  \esd M = n^{-1}\sum_{i=1}^n \delta_{\ev_i(M)},
\end{align*}
where $\delta_x$ is the probability measure concentrated at $x$.  By
Perron-Frobenius theorem (\cite{horn:13:cup}, p.~534), $\sr(M)$ is an
eigenvalue of $M$.  If $M 1_n=1_n$, where $1_n$ is the column vector
of $n$ 1's, then $M$ is called a Markov matrix and $\sr(M)=1$.
Letting $\ev_n(M) = \sr(M)$, $\sev(M)=\max_{i<n} |\ev_i(M)|$ and $1-
\sev(M)$ are known as the second largest absolute eigenvalue and the
spectral gap of $M$, respectively.  For $X\ismat n$, if all the
entries of $a:=X 1_n$ are positive, then its row-normalized version
refers to the Markov matrix $M = D^{-1}_a X$, where $D_a$ denotes the
diagonal matrix whose diagonal equals $a$.  If $X$ is symmetric, then
the Markov chain with transition matrix $M$ and initial distribution
$\pi = a/1'_n a$ is stationary and has the same distribution as its
time reversal, and for this reason $M$ is called reversible relative
to $\pi$.  Moreover, all $\ev_i(M)$ are real as $M$ is similar to
$D^{-1/2}_a M D^{-1/2}_a$, where $D^{1/2}_a$ denotes any symmetric
matrix whose square equals $D_a$.

Let $X_n\ismat n$ be symmetric random matrices with positive entries
almost surely (\as).  Let $M_n$ be its row-normalized version.
Suppose the diagonal and upper diagonal entries of $X_n$ are \iid$\sim
\nu_n$.  If $\nu_n = \nu$ for all $n$, then by \cite
{bordenave:10:alea}, provided that the $4\th$ moment of $\nu$ is
finite, $\sev(M_n)\to 0$ \as as $n\toi$.  On the other hand, if
$\nu$ is in the domain of attraction of a stable law of index in
$(0,2)$, then by \cite {bordenave:11:ap}, $\sev(M_n)\to 1$ \as.  In
either case, the spectral gap of $M_n$ cannot be tuned.  The results
suggest that, in order for the spectral gap or, equivalently,
$\sev(M_n)$ to be tunable, the (marginal) distribution of the entries
of $X_n$ needs to change according to $n$.

Indeed, there are simple solutions along this line.  Given $n\ge 5$,
randomly pick four different numbers $k, l, s$ and $t$ from 1, \ldots,
$n$.  Let $A_n=(\ai_{i j})\ismat n$ with $\ai_{i j} = \cf{\{i,j\} =
  \{k,l\}\text{ or } \{s,t\}}$, where $\cf{\cdot}$ is the indicator
function; $A_n$ is the adjacency matrix of a graph on $n$
vertices with only two edges.  Given $c>0$, let $M_n$ be the
row-normalized version of $c J_n+A_n$, with $J_n\ismat n$ a
matrix of $1/n$'s.  As $\det(z - M_n) = [z + 1/(1+c)]^2 z^{n - 5} [z
- (1-4/n)/(1+c)][z - 1/(1+c)](z-1)$, $\sev(M_n)=1/(1+c)$, so it can
be set at any value in $(0,1)$. 

The main problem with the example is that $M_n$ has few features.  It
is nearly the transition matrix of a chain of \iid random variables
uniformly taking $n$ values.  The lack of features is also reflected
in the ESD of $M_n$, which converges to $\delta_0$ as $n\toi$.
Despite this, the example shows that it is possible to tune the
spectral gap by using sparse random graphs.  In general, let $A_n$ be
the adjacency matrix of a random graph.  Define the row-normalized
version of $c J_n + A_n$ as
\begin{align} \label{e:matrices}
  M_n = D^{-1}_n (c J_n + A_n) \quad\text{with}\quad 
  D_n =D_{c 1_n + A_n 1_n}.
\end{align}
Although $A_n$ can be highly reducible, $M_n$ is always irreducible
and aperiodic and so $\sev(M_n)<1$.  Since all the eigenvalues of
$M_n$ are real, we always assume that they are sorted as
\begin{align*}
  -1< \ev_1(M_n)\le \ldots \le \ev_{n-1}(M_n) < \ev_n(M_n)=1.
\end{align*}
Then $\sev(M_n) = \max(|\ev_1(M_n)|, |\ev_{n-1}(M_n)|)$.  We simply
call $M_n$ reversible, as there is only one stationary distribution
associated with it.  The closely related matrix $I_n - D^{-1/2}_n (c
J_n + A_n) D^{-1/2}_n$ is known as a normalized Laplacian regularized
by $c$.  The effects of $c$ on spectral clustering and concentration
of the ESD have been studied in statistical machine learning
\cite{joseph:13:arXiv, le:15:arXiv}.

The close relation between random matrices and random graphs is well
known; see \cite{mckay:81:laa, tran:13:rsa, dumitriu:12:ap,
  bordenave:11:ap, bordenave:14:cpam, jiang:12:ptrf,
  narayan:12:isccsp, khorunzhy:01:advap, jiang:12:rmta, wood:12:aap,
  tao:08:ccm, gotze:10:ap} and references therein.  In \cite
{mckay:81:laa}, it is shown that if $A_n$ is the adjacency 
matrix of a random graph following the uniform distribution
$\drg_{n,d}$ on the set of regular graphs on $n$ vertices with fixed
degree $d\ge 2$, then as $n\toi$, $\esd{A_n}$ weakly converges \as
with limiting density $f(x) = d (4 d-4-x^2)^{1/2}_+/[2\pi (d^2-x^2)]$,
where $a_+ := \max(a,0)$.  By Weyl's inequality, $\ev_i(D^{-1}_n A_n)
\le \ev_i(M_n) \le \ev_{i+1}(D^{-1}_n  A_n)$ for $i<n$ (cf.\ 
\eqref{e:interlacing}).  Consequently, $M_n$ and $D^{-1}_n A_n =
(c+d)^{-1} A_n$ have the same limiting ESD density $(c+d)
f((c+d)x)$, whose support is the interval between $\pm 2 \sqrt{d-1}/(c
+d)$.  Thus $\sev(M_n)$ is asymptotically lower bounded by
$2\sqrt{d-1}/ (c+d)$.  On the other hand, by the above Weyl's
inequality and the fact that $\sr(A_n)$ is less than the maximum row
sum of $A_n$ (\cite {horn:13:cup}, p.~345--347), $\sev(M_n)\le
\sr(A_n)/ (c+d) \le d/(c+d)$.  In particular, when $d=2$, in which
case the graph consists of disjoint cycles, $\sev(M_n) \to 2/(c+2)$
\as.  Also, under various distributions on regular multigraphs
of fixed degree $d$ that allow multiple edges and, in some cases,
self-loops, for any fixed $l$, $|\ev_l(A_n)|$ and $\ev_{n-l}
(A_n)$ converge to $2\sqrt{d-1}$ in probability, yielding
$\sev(M_n)\to 2\sqrt{d-1}/(c+d)$ \cite{friedman:08:mams}.  However,
when $d>2$, $\sev(M_n)$ cannot be arbitrarily tuned as it is
asymptotically upper bounded by $2 \sqrt{d - 1}/d<1$.  Perhaps
important, under any distribution on regular (multi)graphs, since
$M_n$ is doubly Markov, i.e., $M'_n$ is Markov as well, the stationary
distribution associated with $M_n$ is uniform.  If one wishes to
sample a large Markov matrix with a nonuniform stationary
distribution, then a different random graph needs to be exploited.  We
also mention that for a uniformly sampled doubly Markov matrix, which
is irreversible \as, its limiting ESD is degenerate
\cite{nguyen:14:ap}.

We shall consider the row-normalized version $M_n$ of $c J_n + A_n$
with $A_n$ the adjacency matrix of a random graph following $\drg(n,
p/n)$, the distribution on graphs on $n$ vertices such that each pair
of vertices is connected by an edge with probability $p/n$,
independently from the other pairs (\cite{bollobas:98:sv-ny}, VII).
It is easy to see that for large $n$, the stationary distribution 
associated with $M_n$ is nonuniform with high probability.  We shall
fix $c>0$ and $p>0$ when deriving the asymptotic spectral properties
of $M_n$.  It is known that for both $\drg_{n,d}$ and $\drg(n,p/n)$,
if $d\toi$ and $p\toi$ as $n\toi$, then the ESD of suitably scaled and
centered $A_n$ tends to the semi-circle law \cite{tran:13:rsa,
  dumitriu:12:ap}.  It is also known that when $p>1$ is fixed, the
adjacency matrix of the giant component of a $\drg(n,
p/n)$-distributed graph has a spectral gap asymptotically equal to 0
\cite{narayan:12:isccsp}.  However, these results provide no
indication on the spectral properties of $M_n$.

For the rest of the note, denote
\begin{align*}
  \tau_c = 1/\sqrt{1+c}, \quad c\ge 0.
\end{align*}
One of the main results of the note is the following.
\begin{theorem} \label{t:gap}
  Fix $c>0$ and $p>0$.  For $n>p$, let $A_n$ be the adjacency matrix
  of a random graph following $\drg(n, p/n)$.  Let $k = \lfloor
  p\rfloor + 1$.  Fix $l\ge 1$ and $0<b<\tau_c$.  Then $\pr\{b\le
  \ev_{n-l}(M_n)\le \tau_{c/k} \text{ 
    and } -\tau_{c/k} \le \ev_l(M_n) \le -b\}\to 1$ as $n\toi$.
\end{theorem}

Thus, roughly speaking, $\sev(M_n)$ asymptotically lies between
$\tau_c$ and $\tau_{c/k}$.  In particular, if $p\in (0,1)$, then
$\sev(M_n)\to \tau_c$ in probability.  To prove Theorem \ref {t:gap},
in Section \ref{s:upper}, we show that $\sev(M_n)$ is asymptotically
dominated by $\tau_{c/k}$.  Then, in Section \ref{s:esd}, we show that
$\esd{M_n}$ weakly converges in probability to a symmetric
nondegenerate distribution and characterize the moments of the
limiting distribution in terms of a random walk on a Galton-Watson
tree.  The proof uses the local convergence of random graphs
\cite{bordenave:11:ap, bordenave:11b:ap}.  In Section \ref{s:esssup},
we show that the essential supremum of the limiting distribution is
$\tau_c$, which together with the result in Section \ref{s:upper}
proves Theorem \ref {t:gap}.  In this section we also report some
numerical results which suggest that bounds for $\sev(M_n)$ are not
tight, especially the upper bound when $p$ is large.  Finally, in
Section \ref{s:mnt}, we provide a more explicit formula for the
moments of the limit of $\esd{M_n}$, using the standard moment method.
Some of the results in previous sections can also be established by
the method \cite{chi:15:arXiv}.

\subsection{Notation}
Following \cite {bollobas:98:sv-ny}, a (labeled) graph $G$ has no
multiple edges or self-loops, and all its edges are undirected.
Denote by $V(G)$ and $E(G)$ the vertex set and edge set of $G$,
respectively.  Each $e\in E(G)$ is an unordered pair $\{u,v\}$, with
$u\ne v\in V(G)$; $u$ is called an endpoint of $e$, denoted $u\in e$.
When direction has to be taken into account, denote by $(u,v)$ the
directed edge starting at $u$ and ending at $v$.  The adjacency matrix
of $G$ is $A_G = (\ai_{u v})_{u,v\in G}$ with $\ai_{u v} = \cf{\{u,v\}
  \in E(G)}$.  Denote by $|A|$ the cardinality of a set $A$.  Denote 
$|G|=|V(G)|$ and $e(G)=|E(G)|$, and refer to them as the order and
size of $G$, respectively.  For brevity, denote $u\in G$ if $u\in
V(G)$.  Denote by $d(u,G) := |\{e\in E(G): u\in e\}|$ the degree of
$u\in G$. If $G'$ is another graph, denote by $G\cup G'$ the graph
with vertex set $V(G)\cup V(G')$ and edge set $E(G)\cup E(G')$, and
denote $G\sim G'$ if the two graphs are isomorphic (\cite
{bollobas:98:sv-ny}, p.~3).  If $v\in G$ and $v'\in G'$, and if there
is a graph isomorphism $\sigma: G\to G'$, such that $\sigma(v) = v'$,
then $(G,v)$ and $(G',v')$ are called isomorphic rooted graphs (rooted
with $v$ and $v'$, respectively).  For a finite set $I$, denote by
$1_I$ the column vector of $1$'s indexed by $I$.  For $k\ge 1$, a path
on $I$ of length $k$ is a sequence $\vf v=(\eno v {k+1})$ with
$v_i\in I$ and $v_i \ne v_{i+1}$ 
for $i\le k$; note the requirement that adjacent $v_i$'s be different.
If $v_{k+1} = v_1$, then $\vf v$ is said to be closed.

For properties of $\drg(n,a)$, see \cite{bollobas:98:sv-ny,
  bollobas:01:cup}.  For $a\in [0,1]$, denote by $\dbern(a)$ the
Bernoulli distribution with mass $a$ on 1.  Denote by $\dpois(p)$ the
Poisson distribution with mean $p\ge 0$.  The essential supremum of a
measure $\nu$ on $\Reals$ is $\ess\sup\nu = \sup\{x:
\nu(x,\infty)>0\}$.  For $M\ismat n$ and $k\ge 0$, denote by
$\mnt_k(M)$ the $k\th$ moment of $\esd M$, which equals
$(1/n)\tr(M^k)$ (\cite {bai:10:sg-ny}, Eq.~(1.3.2)).

\section{Upper bound of spectral radius} \label{s:upper}
Fix $c>0$ and $p>0$.  For $n>p$, let $A_n = A_G$ with $G\sim \drg(n,
p/n)$.  Define $M_n$ and $D_n$ by \eqref{e:matrices}.  The spectrum
of $M_n$ is identical to that of $D^{-1/2}_n (c J_n + A_n)
D^{-1/2}_n$.  Since $D^{-1/2}_n J_n D^{-1/2}_n$ is of rank one with
the only nonzero eigenvalue being positive, by Weyl's inequality
(\cite {horn:13:cup}, Corollary 4.3.3)
\begin{align}  \label{e:interlacing}
  \ev_i(D^{-1}_n A_n) \le \ev_i(M_n) \le \ev_{i+1}(D^{-1}_n A_n),
  \quad 1\le i<n.
\end{align}
Consequently, to prove the bound involving $\tau_{c/k}$ in
Theorem \ref{t:gap}, i.e., given $l\ge 1$, $\pr\{\ev_{n-l}(M_n)\le
\tau_{c/k}$ and $\ev_l(M_n) \ge -\tau_{c/k}\}\to 1$, it suffices to
prove the following.
\begin{prop}  \label{p:upper}
  Let $p>0$ and $k=\lfloor p\rfloor + 1$.  Then $\pr\{\sr(D^{-1}_n
  A_n)\le \tau_{c/k}\}\to 1$ as $n\toi$. 
\end{prop}

For graph $G$, denote
\begin{gather} \label{e:K}
  K_G = K_G(c) = D^{-1}_{c 1_{V(G)} + A_G 1_{V(G)}} A_G
\end{gather}
and analogously $K_n = D^{-1}_n A_n$.  Put $q = \sr(K_G)$.  By
Perron-Frobenius theorem (\cite {horn:13:cup}, p.~534) $q =
\ev_{|G|}(K_G)$ and if $|G|>1$ and $G$ is connected, then $q>0$ and
there is a vector $f=(f(u))_{u\in G}$ with all $f_u> 0$, such that
\begin{align}  \label{e:PF}
  K_G f = q f.
\end{align}
Denote by $N(u)$ the neighborhood of $u$ in $G$, i.e., the set of $v\in
G$ with $\{u,v\}\in E(G)$.
\begin{lemma} \label{l:sr}
  Let $G$ be connected with $|G|>1$ and $C\ne\emptyset$ be a subset
  of $V(G)$.  Denote by $h(u)$ the distance of $u\in G$ to $C$.
  Define $\omega(u) = f(u) q^{h(u)}$.  For $i=0,\pm 1$, define $N_i(u)
  = \{v\in N(u): h(v) = h(u)+i\}$ and $d_i(u) = |N_i(u)|$.  Then
  \begin{align} \label{e:SR}
    q^{-1} \sum d_0(u) \omega(u) + q^{-2} \sum d_{-1}(u) \omega(u)
    =
    \sum [d_0(u) + d_{-1}(u) + c] \omega(u).
  \end{align}
\end{lemma}
\begin{proof}
  Since $N_0(u)$, $N_-(u)$, and $N_{+1}(u)$ partition $N(u)$,
  \eqref{e:PF} can be written as
  \begin{align*}
    \sum_{v\in N_{-1}(u)} f(v) + \sum_{v\in N_0(u)} f(v) + 
    \sum_{v\in N_{+1}(u)} f(v)
    =
    q[c+d(u,G)] f(u).
  \end{align*}
  Multiplying both sides by $q^{h(u)-1}$ yields
  \begin{align*}
    \sum_{v\in N_{-1}(u)} \omega(v) +
    q^{-1}\sum_{v\in N_0(u)} \omega(v)
    + q^{-2} \sum_{v\in N_{+1}(u)}\omega(v)
    =
    [c+d(u,G)] \omega(u).
  \end{align*}
  Take sum over $u$.  Since $v\in N_{-1}(u)\Iff u\in N_{+1}(v)$
  and $v\in N_0(u)\Iff u\in N_0(v)$,
  \begin{align*}
    \sum_u \sum_{v\in N_{-1}(u)} \omega(v)
    = 
    \sum_v \sum_{u\in N_{+1}(u)} \omega(v)
    =
    \sum_v d_{+1}(v)\omega(v)
  \end{align*}
  and likewise,
  \begin{align*}
    \sum_u \sum_{v\in N_0(u)} = \sum_v d_0(v) \omega(v), \quad
    \sum_u \sum_{v\in N_{+1}(u)} = \sum_v d_{-1}(v) \omega(v).
  \end{align*}
  Combining the equations and noticing $d(u,G) = d_0(u) +
  d_{-1}(u) + d_{+1}(u)$, \eqref{e:SR} then follows.
\end{proof}

\begin{lemma} \label{l:SR-tree}
  Let $G$ be a connected graph.  If $G$ is a tree or a unicyclic
  graph, then
  \begin{align} \label{e:SR-upper}
    \sr(K_G) < \tau_c.
  \end{align}
  Furthermore, if $G$ is a unicyclic graph, then
  \begin{align} \label{e:SR-lower}
    \sr(K_G) \ge (1+c/2)^{-1} \quad\text{with
      ``$=$'' $\Iff G$ is a cycle}.
  \end{align}
\end{lemma}

\begin{proof}
  First, let $G$ be a tree.  If $|G|=1$, then $K_G =0$ and \eqref
  {e:SR-upper} is trivial.  Let $|G|\ge 2$.  Pick an arbitrary vertex
  $\theta\in G$ and let $C=\{\theta\}$.  It is easy to see that for
  any $u\in G$, $d_0(u)=0$ and $d_{-1}(u)=\cf{u\ne \theta}$.  Then
  \eqref{e:SR-upper} follows from \eqref{e:SR}, which now takes the
  form
  \begin{align} \label{e:SR-tree}
    q^{-2} \sum_{u\ne\theta} \omega(u)
    =
    (1+c)\sum_{u\ne\theta}\omega(u) + c \omega(\theta).
  \end{align}
  Next, let $G$ be unicyclic.  Let $C$ be the cycle subgraph of $G$.
  Then $|C|\ge 3$.  The subgraph of $G$ obtained by removing the edges
  in $C$ consists of $|C|$ isolated trees, each containing exactly one
  vertex in $C$.  It can be seen that $d_0(u) = 2\cf{u\in C}$ and
  $d_{-1}(u) = \cf{u\not\in C}$.  Then by
  \eqref{e:SR},
  \begin{align*}
    (2/q)\sum_{u\in C} \omega(u) + q^{-2} \sum_{u\not\in C} \omega(u)
    =
    (c+2)\sum_{u\in C}\omega(u) + 
    (c+1)\sum_{u\not\in C}\omega(u).
  \end{align*}
  If $G$ is a cycle, then $C=G$ and the equation yields $q=1/(1+c/2)$.
  If $G$ is not a cycle, then $\sum_{u\not\in C} \omega(u)>0$.  If
  $q\le 1/(1+c/2)$, then from $2/q\ge c+2$ and $\sum_{u\in C}
  \omega(u)> 0$, it follows that $c+1\ge q^{-2}$, or $q\ge \tau_c >
  1/(1+c/2)$, which is a contradiction.  Thus $q>1/(1+c/2)$.  But then
  $2/q<c+2$, implying $q^{-2}>c+1$, or $q< \tau_c$.
\end{proof}

It may be worth noting that if $G$ is a tree, then $|G|\toi$ does
not guarantee that $\sr(K_G)\to \tau_c$.  For example, suppose $d(v,
G)<1+c$ for all $v\in G$.  Put $d_0 = \lceil c\rceil$.  Then $d(v,
G)\le d_0$.  Let $f$ be as in \eqref{e:PF} and $\theta = \arg\max
f(v)$.  Then for $k\ge 1$, $\sum_{h(u) = k} \omega(u) \le d^k_0 q^k
f(\theta) < [d_0/(1+c)]^k \omega(\theta)$, giving $\sum_{u\ne \theta}
\omega(u) \le b \omega(\theta)$ with $b = \sum_k [d_0/(1+c)]^k <
\infty$.  Then by \eqref{e:SR-tree}, $q\not\to \tau_c$.

\begin{proof}[Proof of Proposition \ref{p:upper}]
  By definition, $K_n = D^{-1}_n A_n = K_G$ with $G\sim \drg(n,p/n)$.
  First, suppose $0<p<1$.   Write the connected components of $G$
  as $\eno G s$.  Then $K_G$ can be partitioned as
  \begin{align*}
    K_G
    =
    \begin{pmatrix}
      K_{G_1}  & & \\
      & \ddots & \\
      & & K_{G_s}
    \end{pmatrix}.
  \end{align*}
  The eigenvalues of $K_G$ therefore are exactly those of
  $K_{G_i}$, counting multiplicity.  Since $0<p<1$, $\pr\{$all
  $G_i$ are trees or unicyclic graphs$\}\to 1$ as $n\toi$ (\cite
  {bollobas:01:cup}, Corollary 5.8).  This combined with Lemma
  \ref{l:SR-tree} yields $\pr\{\sr(K_n) \le \tau_c\}\to 1$.

  To continue, note that given $0<p_0<p_1<1$, as $n\toi$,
  \begin{align} \label{e:p-uniform}
    \inf_{p_0\le p\le p_1}
    \pr\nolimits_p\{\sr(K_n) \le \tau_c\}\to 1,
  \end{align}
  where $\pr_p$ denotes probability under $\drg(n, p/n)$.  Indeed,
  from the proof of Theorem 5.7 and Corollary 5.8 in
  \cite{bollobas:01:cup}, as $n\toi$, $\inf_{p_0\le p\le
    p_1}\pr\{$every component of $G$ is a tree or a unicyclic
  graph$\}\to 1$.  Then \eqref{e:p-uniform} follows from the same
  argument for the already-proved case $0<p<1$.

  Now let $p\ge 1$.  Then $k:=\lfloor p\rfloor + 1>1$.  For $n>p$, let
  $T_{1,n}$, \ldots, $T_{k,n}$ be \iid$\sim A_G$ with $G\sim \drg(n,
  p'_n/n)$, where
  \begin{align*}
    p'_n = \frac{p}{k-(k-1)p/n}.
  \end{align*}
  Since $p'_n\in (0,n)$, $G$ is well defined.  Let $T_n = (t_{i j}) =
  \sum_{s=1}^k T_{s,n}$.  Since $t_{i j}$, $i<j$, are \iid, for any
  $B = (b_{i j})\in \{0,1\}^{n\times n}$ with $b_{i j} = b_{j i}$ and
  $b_{i i}=0$,
  \begin{align*}
    \pr\{T_n = B\gv T_n\in \{0,1\}^{n\times n}\}
    &=
    \frac{\pr\{t_{i j}=b_{i j},\, i<j\}}
    {\pr\{t_{i j}\in \{0,1\},\, i<j\}}
    =
    \prod_{i<j} \frac{\pr\{t_{i j}=b_{i j}\}}
    {\pr\{t_{i j}\in \{0,1\}\}}.
  \end{align*}
  Since $\pr\{t_{i j}=0\} = (1-p'_n/n)^k$ and $\pr\{t_{i j} = 1\} = k
  (p'_n/n)(1-p'_n/n)^{k-1}$, direct calculation shows that conditional
  on it being in $\{0,1\}^{n\times n}$, $T_n$ has the same
  distribution as $A_n$.  For $i<j$, $\pr\{t_{i j}\in\{0,1\}\} \ge 1
  -[k(k-1)/2] (p'_n/n)^2$.  On the other hand,  $p'_n\to p/k$ as
  $n\toi$.  Then for $n$ large enough, $\pr\{T_n \in \{0,1\}^n\} >
  \exp(-p^2)$, so letting $\Delta_n = D_{c1_n  + T_n 1_n}$ and
  $C=\exp(p^2)$, for any $x$,
  \begin{align} \label{e:decomp}
    \pr\{\sr(K_n) > x\} 
    &= 
    \pr\{\sr(\Delta^{-1}_n T_n)> x\gv T_n \in
    \{0, 1\}^{n\times n}\}
    \nonumber \\
    &\le
    C\pr\{\sr(\Delta^{-1/2}_n T_n \Delta^{-1/2}_n)> x\}.
  \end{align}
  Put $\Delta_{s,n} = D_{c 1_n /k +  T_{s,n} 1_n}$ and $B_{s,n} =
  \Delta^{-1/2}_{s,n} T_{s,n} \Delta^{-1/2}_{s,n}$.  Then
  $\Delta_n = \Delta_{1,n} + \cdots + \Delta_{k,n}$ and
  \begin{align*}
    \Delta^{-1/2}_n T_n \Delta^{-1/2}_n
    =
    \sum_s \Delta^{-1/2}_n \Delta^{1/2}_{s,n} B_{s,n}
    \Delta^{1/2}_{s,n} \Delta^{-1/2}_n.
  \end{align*}

  Fix an arbitrary $a\in (p/k, 1)$.  For $n$ large enough, $p'_n\in
  [p/k, a]$.  Then by \eqref{e:p-uniform}, the probability of the
  event that $\sr(B_{s,n})\le \tau_{c/k}$ for all $1\le s\le k$ tends
  to 1.  On this event, for any $u\in\Reals^n$,
  \begin{align*}
    |u'\Delta^{-1/2}_n T_n \Delta^{-1/2}_n u|
    &\le
    \sum_s |u'\Delta^{-1/2}_n \Delta^{1/2}_{s,n} B_{s,n}
    \Delta^{1/2}_{s,n} \Delta^{-1/2}_n u|
    \\
    &\le
    \sum_s \sr(B_{s,n}) |\Delta^{1/2}_{s,n} \Delta^{-1/2}_n u|^2
    \\
    &\le
    \tau_{c/k} \sum_s |\Delta^{1/2}_{s,n} \Delta^{-1/2}_n u|^2
    =
    \tau_{c/k} |u|^2.
  \end{align*}
  It follows that $\pr\{\sr(\Delta^{-1/2}_n T_n \Delta^{-1/2}_n)\le
  \tau_{c/k}\} \to 1$, so by \eqref{e:decomp}, $\pr\{\sr(D^{-1}_n A_n)
  \le \tau_{c/k}\} \to 1$.
\end{proof}

\section{Convergence of ESD} \label{s:esd}
Let $A_n$, $M_n$, $D_n$, and $K_n = D^{-1}_n A_n$ be as in previous
sections.  We shall show that $\esd{M_n}$ weakly converges as $n\toi$.
From Weyl's inequality \eqref{e:interlacing}, $\esd{M_n}$ weakly
converges in probability (resp.\ \as) $\Iff\esd{K_n}$ does so in
probability (resp.\ \as) and, provided the convergence holds, the two
ESDs have the same limit.  Therefore, we shall focus on $\esd{K_n}$
instead.  The approach we shall take is the local convergence of
random graphs; see \cite {bordenave:11b:ap} and references therein,
and \cite {bordenave:11:ap} for extension to the ESD of random
matrices whose entries belong to the domain of attraction of stable
laws.

Let $G$ be a graph and $v_0\in G$.  Fix $c\ge 0$.  Consider the
following random walk on $G$ starting from $v_0$ at step 0.  If
$d(v_0, G)\ge 1$, then if at step $k\ge 0$ the random walk is at $v$
with $d(v, G)=d (\ge 1)$, then at step $k+1$, it either moves to a
neighbor of $v$ with probability $1/(c+d)$, or is killed with
probability $c/(c+d)$.  If $d(v_0, G)=0$, then the random walk is
killed at step 1, regardless of the value of $c$.  Let
\begin{align*}
  r_k(G, v_0, c) =
  \pr\{\text{the random walk is alive and at $v_0$ at step}\ k\}.
\end{align*}
Let $\Empty$ be an arbitrary element.   Denote by $[(G, v_0)]$ the
class of graphs rooted with $\Empty$ that are isomorphic to $(G,
v_0)$.  Then $r_k(G,v_0,c)$ depends on $(G,v_0)$ only through $[(G,
v_0)]$. 

Recall that in order for $K_n = D^{-1}_n A_n$ to be always well
defined, $c$ has to be strictly positive.  In the following, we
redefine $D_n$ such that its $i\th$ diagonal element is 1 if the
entire $i\th$ row of $A_n$ is 0.  With this definition, $c$ can be 0. 

\begin{theorem} \label{t:esd}
  Let $c\ge 0$.  As $n\toi$, $\esd{K_n}$ weakly converges in
  probability.  The weak convergence is \as if $n$ is replaced with
  any subsequence $n_j$ with $\sum n^{-1}_j<\infty$.  The limiting
  distribution is symmetric and nondegenerate, and for $k\ge 1$, its
  $k\th$ moment is $\mnt_k=\mean r_k(T, \Empty, c)$, where $T$ be a
  random Galton-Watson tree rooted with $\Empty$ and with $\dpois(p)$
  offspring distribution.
\end{theorem}
Note that for any tree $T$, if $k$ is odd, then $r_k(T, \Empty,c)=0$
and hence $\beta_k=0$.  This immediately leads to the symmetry of the 
limiting distribution.

\begin{proof}
  Put $K_n = (x_{ij})$.  Denote by $C_{k,n}$ the set of closed paths
  of length $k$ on $\{1,\ldots, n\}$.  For $\vf i\in C_{k,n}$, denote
  $x(\vf i) = x_{i_1 i_2} x_{i_2 i_3} \cdots x_{i_{k-1} i_k} x_{i_k
    i_1}$.  Then for $k\ge 1$ and $s=1,\ldots,n$, the $s\th$ diagonal
  entry of $K_n^k$ is
  \begin{align*}
    (K^k_n)_{ss}
    =
    \sum_{\vf i\in C_{k,n},\ i_1=i_k=s} x(\vf i).
  \end{align*}
  Since $x(\vf i)$ is the probability that the random walk is alive
  after traversing the closed path $\vf i$,
  \begin{align*}
    (K^k_n)_{ss}
    =
    r_k(G, s, c), \quad s=1,\ldots, n.
  \end{align*}
  For a random walk on $G$ that starts from $s$, if it returns to $s$
  at step $k$, then the vertices it visits by then each has at most
  distance $k-1$ from $s$, and so the neighbors of each such vertex
  has at most distance $k$ from $s$.  Denote by $G_{k, s}$ the
  subgraph of $G$ whose vertex set consists of vertices with distance
  from $s$ no greater than $k$ and whose edge set consists of edges in
  $G$ connecting these vertices.  Then $r_k(G,
  s,c) = r_k(G_{k,s}, s, c)$.  It is well known that, given $s$, as
  $n\toi$, $G$ rooted with $s$ converges locally to $T$ in
  distribution.  This means that for any $k$, $G_{k, s}$ rooted with
  $s$ converges in distribution to $T_k$, the subtree of $T$
  consisting of $\Empty$ and its first $k$ generations of descendants;
  see for example \cite {bordenave:11b:ap}.  As a result,
  $r_k(G_{k,1}, 1,c)\to r_k(T_k, \Empty,c) = r_k(T, \Empty, c)$ in
  distribution.  By the above displays, $\mnt_k(K_n) = n^{-1}
  \sum_{s=1}^n r_k(G, s, c)$.   Then by exchangeability and dominated
  convergence, $\mean\mnt_k(K_n) = \mean r_k(G, 1, c) = \mean
  r_k(G_{k,1}, 1, c) \to \mean r_k(T, \Empty, c)$.

  We need to show that $\mnt_k(K_n)\to \mnt_k$ in probability as
  $n\toi$, and \as if $n$ is replaced with $n_j\toi$ such that $\sum
  n^{-1}_j < \infty$.  Put $\xi_s=r_k(G,s,c)$.  By exchangeability,
  \begin{align} \label{e:var-mnt}
    \var[\mnt_k(K_n)]
    &=
    n^{-1} \var(\xi_1)
    + 2(1-n^{-1})
    \cov(\xi_1, \xi_2)
    \le
    n^{-1} + 2|\cov(\xi_1,\xi_2)|.
  \end{align}
  Let $S_1 = \cf{\text{distance between 1 and 2 in  $G$ is $>2k$}}$.
  Then
  \begin{align*}
    \pr\{S_1=0\}
    &\le
    \sum_{l=0}^{2k-1}
    \pr\{\text{$\exists\eno i {2k-1}$ s.t.\  $\{i_t,i_{t+1}\}\in
      E(G)$, $0\le t<2k$, with $i_0=1$, $i_{2k}=2$}\} 
    \\
    &\le
    \sum_{l=0}^{2k-1} n^{2k-1} (p/n)^{2k} = O_k(1/n),
  \end{align*}
  where $O_k(\cdot)$ denotes that the implicit constant depends only
  on $k$ in addition to the fixed $p$ and $c$.  Note that when
  $S_1=1$, $G_{k,1}$ and $G_{k,2}$ are disjoint.  Let $S_2 =
  \cf{|G_{k, s}|\le n/2, s=1,2}$.  Denote by $\mathrm{dis}(u,v) =$
  distance between $u$ and $v$ in $G$.  By $|G_{0,s}|=|\{u:
  \mathrm{dis}(s,u)=0\}|=1$,
  \begin{align*}
    \mean|\{u: \mathrm{dis}(s,u)=k\}|
    \le
    \mean\Sbr{
      \sum_{v: \mathrm{dis}(s,v)=k-1} \sum_{u=1}^n \cf{\{u,v\}\in
        E(G)}
    }
    \le 
    p \mean|\{u: \mathrm{dis}(s,u)=k-1\}|,
  \end{align*}
  and induction, $\mean|G_{k,s}|\le 1+ p + \cdots + p^k = O_k(1)$.
  Then by Markov inequality, $\pr\{S_2=0\} = O_k(1/n)$.  Let $S=S_1
  S_2$.  Then $\pr\{S=0\}=O_k(1/n)$.  Conditioning on $S=1$, $[(G_{k,
    1}, 1)]$ and $[(G_{k,2},2)]$ are \iid$\sim [(G_{k,1},1)]$
  conditioning on $|G_{k,1}|\le n/2$.  Since for $s=1, 2$, $\xi_s =
  r_k(G_{k,s}, s,c)$ only depends on $[(G_{k,s}, s)]$, $\cov(\xi_1,
  \xi_2\gv S=1)=0$.  By exchangeability $\mean(\xi_1\gv S) =
  \mean(\xi_2\gv S)$, denoted by $h_S$.  Then
  \begin{align*}
    \cov(\xi_1, \xi_2)
    &=
    \mean[\cov(\xi_1, \xi_2\gv S)] + \cov(\mean(\xi_1\gv S),
    \mean(\xi_2\gv S))
    \\
    &=
    \cov(\xi_1,\xi_2\gv S=0)\pr\{S=0\} + \var(h_S)
    \\
    &=
    O_k(1/n) + (h_1 - h_0)^2 \pr\{S=0\}
    \pr\{S=1\} = O_k(1/n),
  \end{align*}
  so by \eqref{e:var-mnt}, $\var[\mnt_k(K_n)] = O_k(n^{-1})$.  This
  implies that $\mnt_k(K_n) - \mean[\mnt_k(K_n)]\to 0$ in probability,
  so $\mnt_k(K_n)\to \mnt_k$ in probability.  Moreover, for $n_j\toi$
  with $\sum n^{-1}_j<\infty$, by Borel-Cantelli lemma,
  $\mnt_k(K_{n_j}) - \mean[\mnt_k(K_{n_j})]\to 0$ \as, giving
  $\mnt_k(K_{n_j}) \to\mnt_k$ \as.

  Since the entries of $K_n$ are nonnegative with row sums no greater
  than 1, $\sr(D^{-1}_n A_n)\le 1$ and hence $\esd{K_n}$ is supported
  in $[-1,1]$.  Meanwhile, by Weierstrass theorem, polynomials are
  dense in $C([-1, 1])$.  Then by the convergence in probability of
  $\mnt_k(K_n)$ and standard results on weak convergence (\cite
  {breiman:92:siam}, Section 8.4), $\esd{K_n}$ weakly converges in
  probability to a probability distribution with support in $[-1,1]$
  and moments $\mnt_k$.  Finally, for any $n_j\toi$ with $\sum
  n^{-1}_j<\infty$, the \as weak convergence of the ESD of $K_{n_j}$ 
  follows from the \as convergence of $\mnt_k(K_{n_j})$.
\end{proof}

\section{Essential supremum of the limit of ESD} \label{s:esssup}
Let $\esd\infty$ be the limiting distribution of $\esd{K_n}$, where
again $K_n = D^{-1}_n A_n$.  The main result of this section is the
following.
\begin{theorem} \label{t:e-sup}
  For any $p>0$, $\ess\sup\esd\infty = \tau_c$.
\end{theorem}

Thus, for fixed $c>0$, as $p\toi$, $\ess\sup\esd\infty$ does not
vanish.  This may be compared to the case where the underlying random
graph follows $\drg_{n,d}$.  By \cite{mckay:81:laa}, the corresponding
essential supremum is $2\sqrt{d-1}/(c+d)$ so, given $c>0$, it tends to
0 as $d\toi$.

\begin{proof}[Proof of Theorem \ref{t:e-sup}]
  Given $s\ge 1$, the probability $q_s$ that $T$ is a tree on
  $\{\Empty, \eno v s\}$ with $E(T) = \{\{\Empty, v_i\},
  i=1,\ldots, s\}$ is positive.  For this $T$ and $k=2m$, it is easy
  to get $r_k(T,\Empty,c) = [s/(c+s)]^m [1/(c+1)]^m$.  Then by Theorem
  \ref {t:esd}, $\mnt_k \ge q_s [s/(c+s)]^m [1/(c+1)]^m$, yielding
  $\ess\sup\esd\infty \ge \sqrt{s/(c+s)} \tau_c$.  Letting $s\toi$
  then gives $\ess\sup\esd\infty\ge \tau_c$.

  To show $\ess\sup\esd\infty \le \tau_c$, it suffices to consider
  $c>0$.  For $k=2m$, arguing as in the proof of Theorem \ref
  {t:esd}, $\mnt_k = \mean r_k(T, \Empty, c) = \mean r_k(T_{m+1},
  \Empty, c) = \mean (K^k_{T_{m + 1}})_{\Empty \Empty}$.  We claim
  that for any finite graph $G$, 
  $v\in G$, and $k\ge 1$,
  \begin{align} \label{e:power-K}
    |(K^k_G)_{vv}| \le \sr(K_G)^k.
  \end{align}
  Together with Lemma \ref{l:SR-tree}, this implies
  $|(K^k_{T_{m+1}})_{\Empty\Empty}|\le \tau^k_c$.  As a result
  $\mnt_k\le \tau^k_c$ and hence $\ess\sup\esd\infty\le\tau_c$, which
  completes the proof.

  To prove \eqref{e:power-K}, suppose $V(G) = \{1,\ldots, n\}$.  Then
  $K_G = D^{-1}_a A_G$, where $a = (\eno a n)' = c 1_n + A_G 1_n$.
  Let $b = (\sqrt{a_1}, \ldots, \sqrt{a_n})'$.  Then $K^k_G = D^{-1}_b
  B^k D_b$, where $B = D^{-1}_b A_G D^{-1}_b$.  Since $D_b$ is
  diagonal, $(K^k_G)_{ii}=(B^k)_{ii}$, $i\le n$.  Since $B$ is
  symmetric, so is $B^k$.  In general, for any symmetric real-valued
  matrix $H$, since $\sr(H) \pm H$ is nonnegative definite, $\max_i
  |H_{ii}| \le \sr(H)$.  Thus $|(K^k_G)_{ii}| \le \sr(B^k) = \sr(B)^k
  = \sr(K_G)^k$, as claimed.
\end{proof}

We now can prove Theorem \ref{t:gap}.   Without loss of generality,
let $b\in (0, \tau_c)$ be a continuity point of the distribution
function of $\esd\infty$.  Then by Theorems \ref{t:esd} -- \ref
{t:e-sup}, $\esd{K_n}(J) \to\esd\infty(b,\infty) > 0$ in probability
for $J = (-\infty, -b)$, $(b, \infty)$, and so for any $l\ge 1$,
$\pr\{\ev_l(K_n) \le -b$ and $\ev_{n-l+1}(K_n) \ge b \} \to 1$.
Together with Proposition \ref{p:upper}, this yields
\begin{align*}
  \pr\{b\le \ev_{n-l+1}(K_n)\le \tau_{c/k} \text{
    and } -\tau_{c/k} \le \ev_l(K_n) \le -b\}\to 1.
\end{align*}
Since the convergence holds for all $l\ge 1$, then by Weyl's
inequality \eqref {e:interlacing}, the proof is complete.

\begin{figure}[t]
  \begin{center}
    \begin{tabular}{cc}
      \scalebox{.4}{\includegraphics{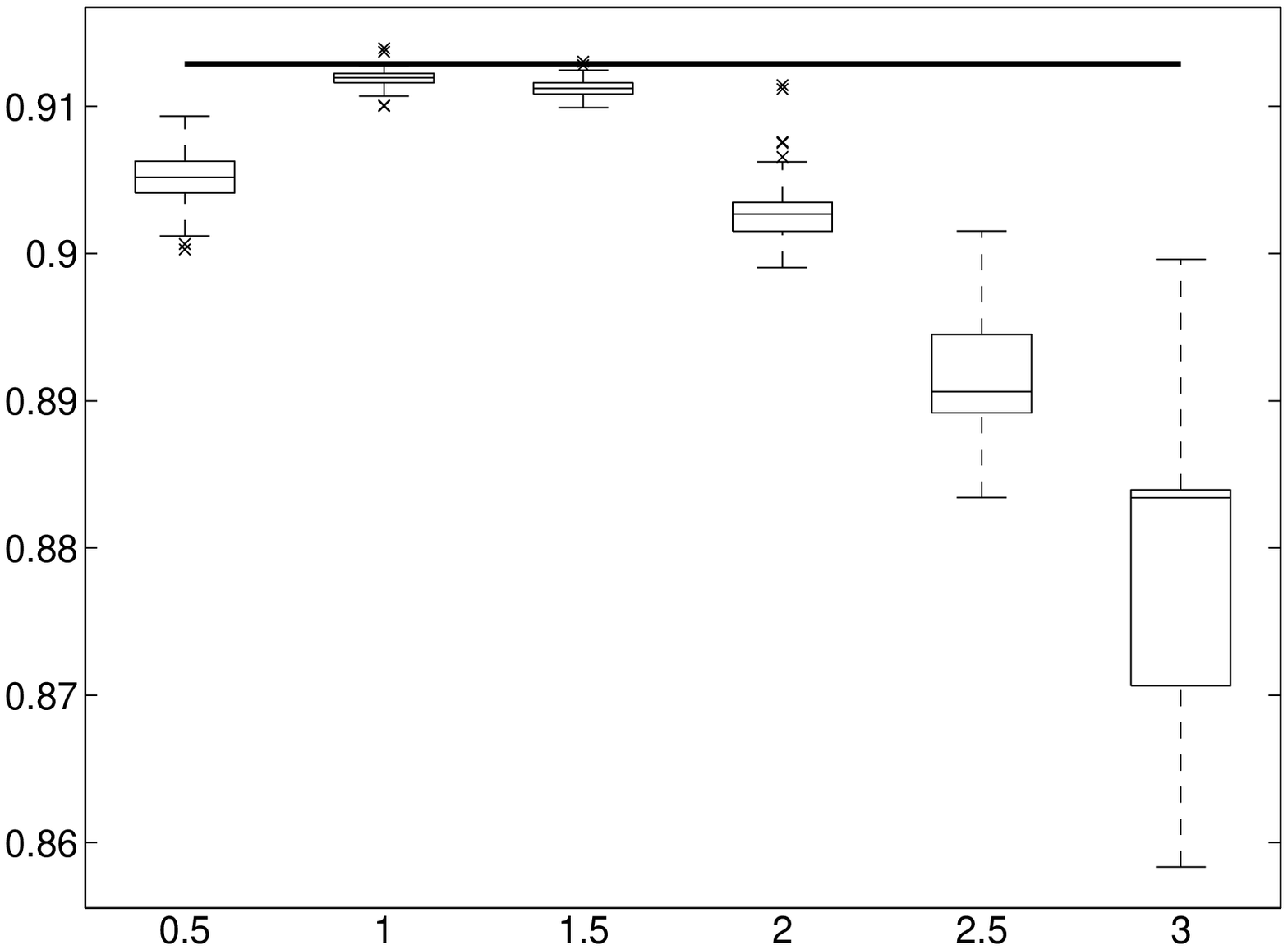}} &
      \scalebox{.4}{\includegraphics{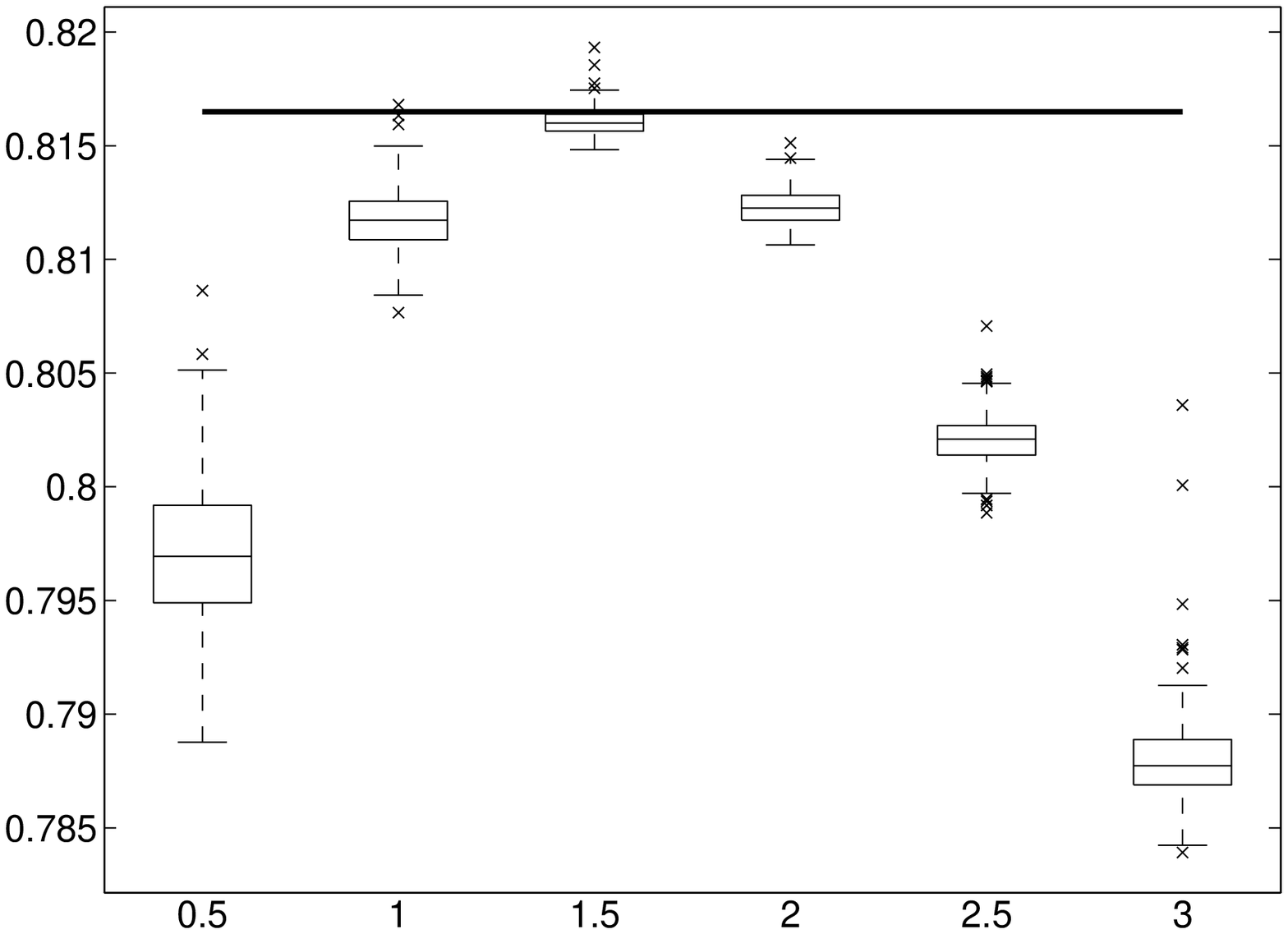}} \\[-4em]
      \hspace{3ex}\framebox[2cm]{$c=0.2$} &
      \hspace{3ex}\framebox[2cm]{$c=0.5$}
    \end{tabular}
    \\[3em]
    \begin{tabular}{c}
      \scalebox{.4}{\includegraphics{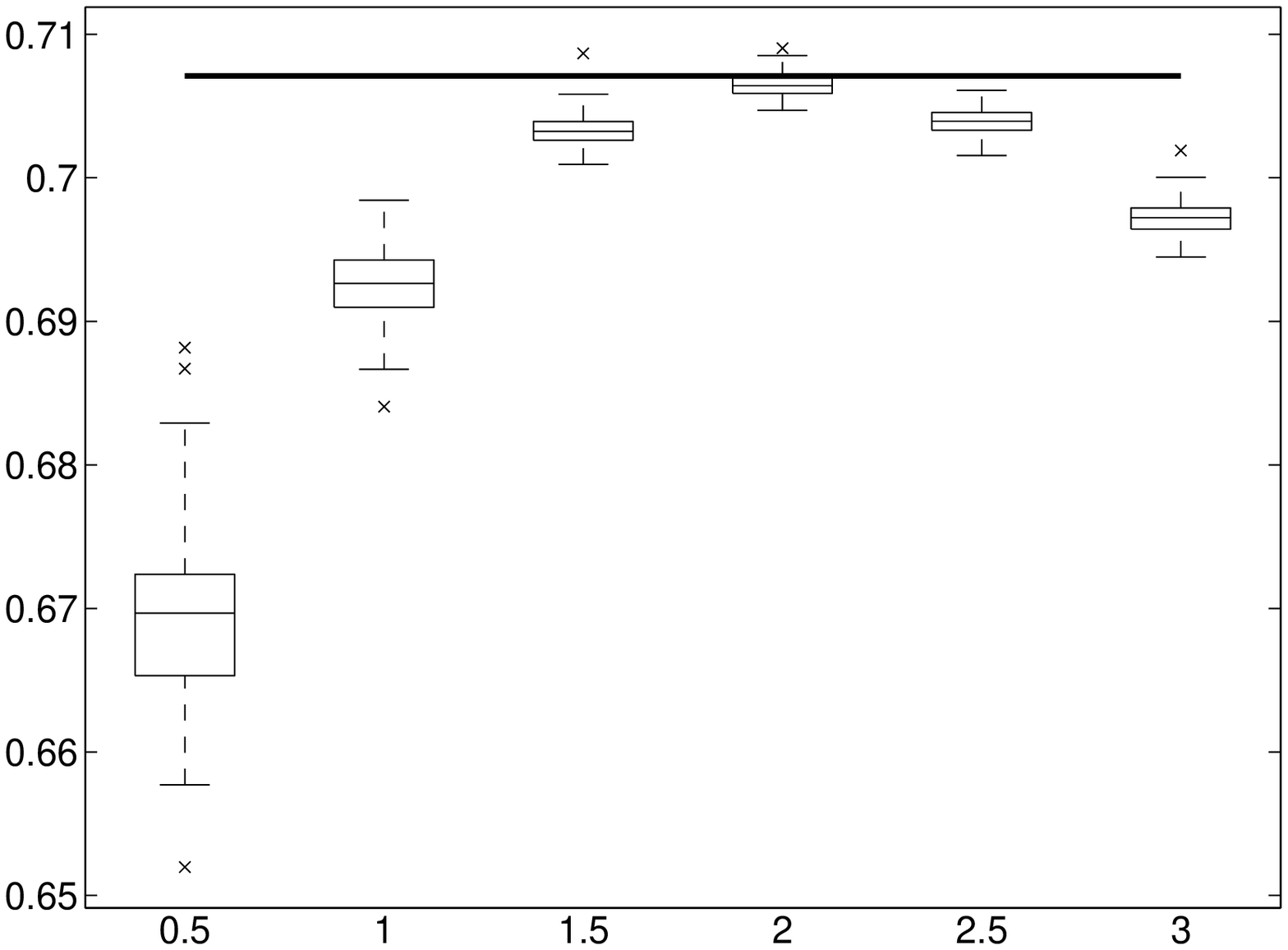}} \\[-4em]
      \hspace{3ex}\framebox[2cm]{$c=1$}\\[2em]
    \end{tabular}
  \end{center}
  \caption{\label{f:simul}
    Boxplots of randomly sampled $\sev(M_n)$.}
\end{figure}

We conducted a simulation study to examine the tightness of the bounds
in Theorem \ref{t:gap}.  Given $c$, for each $p\in \{0.5, 1, 1.5, 2,
2.5, 3\}$, we used MATLAB function \texttt{eig} to calculate
$\sev(M_n)$ for 200 randomly sampled $M_n$ with $n=4000$.  In each
panel of Figure \ref{f:simul}, the boxplot of the sample values of
$\sev(M_n)$ is shown as a function $p$.  On each box, the central mark
is the sample median, the edges of the box are the 1st and 3rd sample
quartiles, the whiskers extend to the most extreme sample values
considered by MATLAB to be non-outliers, and the outliers are plotted
individually as ``x''.  The $y$-coordinate of the long horizon line
extending from $p=0.5$ to 3 equals $\tau_c$.  As Figure \ref{f:simul}
shows, for $p=0.5$ and most of $p\ge 1$, even when $n = 4000$,
$\sev(M_n)$ is still quite below $\tau_c$.  Since $\sev(M_n)$ is
asymptotically lower bounded by $\tau_c$ according to Theorem \ref
{t:e-sup}, this suggests that its convergence is slow.  The plots also
indicate that for different values of $c$, there are different values
of $p$ for which the convergence is fastest in terms of how fast
$\sev(M_n)$ approaches or goes above $\tau_c$ and how fast its 
variation decreases.  Among all the pairs of $c$ and $p$, only $(c, p)
= (0.5, 1.5)$ and $(1, 2)$ generated a significant number of
$\sev(M_n)$ that were no less than $\tau_c$, with the fraction of such
$\sev(M_n)$ equal to 22\% and 20\%, respectively.  When $n$ was
increased to 6000, the fraction changed to 18\% and 22\%,
respectively.  However, the differences in fraction are not
statistically significant.  To see if the relatively high fractions
were due to fluctuations of the bulk of the eigenvalues, we counted
the total number of eigenvalues with absolute values no less than
$\tau_c$.  When $n=4000$, for each pair, there were $200\times 4000 =
8\times 10^5$ absolute eigenvalues in total.  Only 246 and 243 of
them, respectively, were no less than $\tau_c$.  When $n=6000$, the
counts changed to 239 and 246, respectively.  Thus, the fluctuation of
the bulk had little to do with the relatively high percentages of
$\sev(M_n)$ greater than $\tau_c$.

We were unable to go beyond $n=6000$ due
to limited computing capacity.  Nevertheless, the numerical results
suggest that for $p>1$, $\tau_c$ is not a tight lower bound, at least
in the probabilistic sense that there is $t'>\tau_c$, such that
$\pr\{\sev(M_n)\ge t'\}\not\to 0$.  The numerical results also
suggest, more convincingly, that the upper bound in Theorem \ref
{t:gap} is far from being tight, especially for large $p$.  It would
be interesting to see whether $\sev(M_n)$ has a nonrandom limit or
weakly converges to a nondegenerate distribution and in either case,
at what rate of convergence.

\section{A formula for moments of the limit of ESD} \label{s:mnt}

This section gives a more explicit formula for the moments
$\mnt_k$ of the limiting distribution $\esd\infty$.  To state the
result, if $\vf v$ is a path on $I_n:=\{1,\ldots, n\}$, denote by
$[\vf v]$ the graph whose vertex set consists of the distinct elements
among $v_i$, and whose edge set consists of the distinct unordered
pairs among $\{v_i, v_{i+ 1}\}$, $1\le i\le k$.  Denote by $C_{k,n}$
the set of closed paths of length $k$ on $I_n$.  Denote by
$\cnt(\cdot, \vf v)$ the number of times an object appears in $\vf v$.
Thus, for $u, u'\in I$ and $e=\{u, u'\}$, $\cnt(u, \vf v) =
\sum_{i=1}^{k+1} \cf{v_i = u}$ and $\cnt(e, \vf v) = \sum_{i=1}^k
\cf{\{v_i, v_{i+1}\}   = e}$.  Also, denote $\cnt_+(u,\vf
v)=\sum_{x\in I} \cnt((u,x),\vf v)$, i.e., the number of directed
edges in $\vf v$ starting at $u$.  Denote by $E_n$ the set of edges
of the complete graph on $I_n$.  Following the definition on p.~17 of
\cite{bai:10:sg-ny}, a path $\vf i \in C_{k,n}$ is called canonical if
$i_1=1$ and $i_j \le \max(\eno i {j-1})+1$ for $2\le j\le k$.  For
such a path $\vf i$, if $|[\vf i]|=t$, then the set of distinct values
of $i_j$ is $\{1, \dots, t\}$.  Let
\begin{align*}
  \Gamma_{k,t}
  =
  \{\vf i\in C_{k,n}: \vf i\text{ is canonical},
  \ [\vf i] \text{ is a tree of size } t\}.
\end{align*}
As long as $n\ge t$, the definition is independent of $n$.  Note that
for $\vf i\in C_{k,n}$, $[\vf i]$ is a tree $\Iff |[\vf i]| = e([\vf
i])+1$, and when this is the case, each $e\in E([\vf i])$ is traversed
by $\vf i$ on both directions the same number of times, and hence $k$
is even.  Therefore, $\Gamma_{k,t}=\emptyset$ if $k$ is odd.

\begin{prop} \label{p:esd-mnt}
  Let $c>0$.  Then for even $k=2m$,
  \begin{align*}
    \mnt_k=
    \sum_{t=2}^{m+1} p^{t-1}\sum_{\vf i\in \Gamma_{k,t}}
    \prod_{a=1}^t
    \mean\Sbr{
      (c+ d(a, [\vf i])+\xi)^{-\cnt_+(a,\vf i)}
    }, \quad\xi\sim \dpois(p).
  \end{align*}
\end{prop}

\begin{proof}
  Given $n$, write $A_n = (\ai_{i j}) \in \{0,1\}^{n\times n}$.  For
  $e=\{i,j\}\in  E_n$, denote $\ai_e =\ai_{i j} = \ai_{j i}$.  Put
  \begin{align*}
    w_i = c + \sum_{j=1}^n \ai_{i j},\quad x_{i j} = \ai_{i j}/w_i.
  \end{align*}
  For $(y_{i j})\in\Reals^{n\times n}$ and $(z_i) \in  \Reals^n$, and
  for $\vf i\in C_{k,n}$, denote $y(\vf i) = y_{i_1 i_2} y_{i_2 i_3}
  \cdots y_{i_{k-1} i_k} y_{i_k i_1}$ and $z(\vf i) = z_{i_1} \cdots
  z_{i_k}$.  Then from the proof of Theorem \ref {t:esd},
  \begin{align} \label{e:trace-path}
    \mean[\tr(K^k_n)]
    =
    \sum_{\vf i\in C_{k,n}} \mean\Sbr{x(\vf i)}.
  \end{align}
  Given $\vf i\in C_{k,n}$, let $t=|[\vf i]|$ and $s = e([\vf i])$.
  Since $\ai(\vf i)= \prod_{e\in E([\vf i])}\ai_e^{\cnt(e, \vf i)}$
  with all $\cnt(e, \vf i)\ge 1$, $x(\vf i) = \ai(\vf i)/w(\vf
  i)\ne 0\Iff\ai_e=1$ for all $e\in E([\vf i])\Iff \ai(\vf i)=1$.  As
  $\ai_e$, $e\in E([\vf i])$, are \iid$\sim\dbern(p/n)$, $\pr\{\ai(\vf
  i)=1\}=(p/n)^s$.  For $j\le k$, $w_{i_j} \ge c+\ai_{i_j i_{j+1}}$
  and $\{i_j, i_{j+1}\} \in E([\vf i])$.  Consequently, $\ai(\vf i)=1$
  implies $w_{i_j}\ge  c+1$ for all $j$.  As a result, 
  \begin{align*}
    \mean[x(\vf i)]
    =
    \mean\Sbr{\lfrac{\cf{\ai(\vf i)=1}}{w(\vf i)}}
    \le
    (c+1)^{-k} \pr\{\ai(\vf i)=1\}
    =
    (c+1)^{-k} (p/n)^s.
  \end{align*}
  Since $[\vf i]$ is connected, $1\le e([\vf i])\le k$ and $2\le |[\vf
  i]|\le e([\vf i])+1$.  For $2\le t\le k$, the number of $\vf i\in
  C_{k,n}$ with $|[\vf i]|=t$ is less than $\binom{n}{t} t^k$.  As a
  result, for $n\ge 2$,
  \begin{align*}
    \sum_{\vf i\in C_{k,n}:\; |[\vf i]|\le e([\vf i])}
    \mean\Sbr{x(\vf i)}
    &=
    \sum_{s=1}^k \sum_{t=2}^s
    \sum_{\vf i\in C_{k,n}:\; |[\vf i]|=t,\, e([\vf i])=s}
    \mean\Sbr{x(\vf i)}
    \nonumber\\
    &\le
    \sum_{s=1}^k \sum_{t=1}^s n^t t^k (c+1)^{-k}(p/n)^s
    \nonumber \\
    &\le
    (c+1)^{-k} \sum_{s=1}^k s^k p^s \sum_{t=1}^s n^{t-s}
    \le
    2 (c+1)^{-k} \sum_{s=1}^k s^k p^s.
  \end{align*}
  Since $p$ and $c$ are fixed, then by \eqref{e:trace-path}
  \begin{align} \label{e:trace}
    \mean[\tr(K^k_n)]
    =
    \sum_{\vf i\in C_{k,n}:\, |[\vf i]| = e([\vf i])+1}
    \mean\Sbr{x(\vf i)} + O_k(1), \quad
    \text{as}\ n\toi.
  \end{align}
  Let $\vf i$ be a path counted on the right hand side of \eqref
  {e:trace} and $|[\vf i]|=t$.  Then $2\le t\le m+1$ and
  \begin{align*}
    x(\vf i)
    =
    \frac{\ai(\vf i)}{w(\vf i)}
    =
    \prod_{a,b=1}^t
    \Grp{\frac{\ai_{ab}}{w_a}}^{\cnt((a,b),\vf i)}.
  \end{align*}
  Clearly $x(\vf i)$ is a deterministic function of $A_n$, denoted by
  $F(A_n)$.  Arrange the elements of $V([\vf i])$ as $\eno z t$ in the
  order of initial appearance in $\vf i$ and let $\sigma(z_l)=l$.
  Then $\sigma: V([\vf i])\to \{1,\ldots, t\}$ is the unique bijection
  such that $\sigma(\vf i) := (\sigma(i_1), \ldots, \sigma(i_k),
  \sigma(i_1))\in \Gamma_{k,t}$.  Extend $\sigma$ to a permutation of
  $\{1,\ldots, n\}$, still denoted $\sigma$.  Let $S = (s_{i j})\ismat
  n$ with $s_{i j} = \cf{i=\sigma(j)}$.  From $S' A_n S = (\tilde
  \ai_{i j})$ with $\tilde \ai_{i j} = \sum_{k l} s_{k i} \ai_{i j}
  s_{l j} = \ai_{\sigma(i) \sigma(j)}$, $x(\sigma(\vf i)) =
  \ai(\sigma(\vf i))/w(\sigma(\vf i)) = F(S' A_n S)$.  Since $A_n\sim
  S' A_n S$, then $x(\vf i)\sim x(\sigma(\vf i))$, in particular,
  $\mean[x(\vf i)] = \mean[x(\sigma(\vf i))]$.
  
  It is easy to see that for each $\vf i\in \Gamma_{k,t}$, there are
  exactly $n!/(n-t)!$ paths counted on the right hand side of
  \eqref{e:trace} that can be mapped in the above way to $\vf i$.  As
  a result, 
  \begin{align} \label{e:mean-em}
    \mean[\tr(K^k_n)]
    =
    \sum_{t=2}^{m+1} \frac{n!}{(n-t)!}
    \sum_{\vf i\in \Gamma_{k,t}}
    \mean\Sbr{x(\vf i)} + O_k(1), \quad
    n\toi.
  \end{align}
  Given $t=2, \ldots, m+1$ and $\vf i\in \Gamma_{k,t}$, for
  $a=1,\ldots,t$, write
  \begin{align*}
    q_a = \sum_{a\in e\in E([\vf i])} \ai_e, \quad
    y_a = \sum_{a\in e\in E_t\setminus E([\vf i])} \ai_e, \quad
    S_a = \sum_{a\in e\in E_n\setminus E_t} \ai_e.
  \end{align*}
  Then $w_a = c+q_a + y_a + S_a$.  Put $y=(\eno y t)$.  Then $y$,
  $\eno S t$, and  $\ai_e$, $e\in E([\vf i])$, are all independent,
  and $x(\vf i)\ne 0\Iff \ai_e=1$ for all $e\in E([\vf i])$.  Since
  $e([\vf i])=t-1$, then
  \begin{align*}
    \mean[x(\vf i)]
    &=
    \mean[x(\vf i)\cf{\ai_e=1 \ \forall e\in E([\vf i])}]
    \\
    &=
    (p/n)^{t-1}
    \mean\Sbr{
      \prod_{a,b=1}^t (c+q_a + y_a + S_a)^{-\cnt((a,b),\vf i)}
      \ \vline\ \ai_e=1\, \forall e\in E([\vf i])
    }.
  \end{align*}
  On the other hand, when $\ai_e=1$ for all $e\in E([\vf i])$,
  $q_a=d(a, [\vf i])$ for all $a=1,\ldots,t$.  Then
  \begin{align*} 
    \mean[x(\vf i)]
    &=
    (p/n)^{t-1}
    \mean\Sbr{
      \prod_{a=1}^t 
      \Grp{
        c+d(a, [\vf i]) + y_a + S_a
      }^{-\sum_{b=1}^t\cnt((a,b),\vf i)}
    }
    \nonumber \\
    &=
    (p/n)^{t-1}
    \mean\Sbr{
      \prod_{a=1}^t (c+d(a, [\vf i])+y_a + S_a)^{-\cnt_+(a,\vf i)}
    }.
  \end{align*}
  Let $n\toi$.  Since $t$ is fixed, $(y, \eno S t)\to (0, \eno \xi
  t)$ in distribution, with $\xi_i$ \iid$\sim \dpois(p)$.  Then by
  \eqref{e:mean-em} and dominated convergence, for $k=2m$,
  \begin{align*}
    \mean[\mnt_k(K_n)]
    =
    n^{-1}  \mean[\tr(K^k_n)]
    \to
    \sum_{t=2}^{m+1} p^{t-1}
    \sum_{\vf i\in \Gamma_{k,t}}
    \mean\Sbr{
      \prod_{a=1}^t (c+d(a, [\vf i])+\xi_a)^{-\cnt_+(a,\vf i)}
    },
  \end{align*}
  finishing the proof.
\end{proof}

\subsubsection*{\bf Acknowledgments}  The author would like to thank
the referees for careful reviews and useful comments, in particular,
their comments on the connection to local convergence and random walk
on a Galton-Watson random tree.

\begin{small}

\end{small}

\end{document}